\documentclass[oneside]{amsart}
\pdfoutput=1
\usepackage{geometry}                %
\usepackage{amssymb,amsmath,latexsym}

\usepackage[utf8]{inputenc}
\usepackage[T1]{fontenc}
\usepackage{lmodern}
\usepackage[dvipsnames]{xcolor}

\usepackage{mathtools}
\usepackage{tikz}
\usetikzlibrary{positioning,calc,patterns}

\definecolor{col1}{RGB}{254,97,0}
\definecolor{col2}{RGB}{100,143,255}

\usepackage{ifpdf}

\usepackage{hyperref}

\newcommand{\Q}{{\mathbb{Q}}}

\newcommand{\Aut}{{\mathrm{Aut}}}  %
\newcommand{\Out}{{\mathrm{Out}}}  %
\newcommand{\Stab}{{\mathrm{Stab}}}

\newcommand{\dd}{{\mathrm{d}}}

\newcommand\vertex[1]{\fill #1 circle (.075)}
\newcommand\rootvertex[1]{\draw[fill=white] #1 +(-3pt,-3pt) rectangle +(3pt,3pt)}

\newcommand{\hecycle}[1]
{
  \foreach \x in {0,...,#1}
{
  \draw [thick] (360/#1*\x:.25) to (360/#1*\x:1);
}
}

\newtheorem{proposition}{Proposition}[section]
\newtheorem{definition}[proposition]{Definition}

\newtheorem*{theorem*}{Theorem}
\newtheorem{lemma}[proposition]{Lemma}

\theoremstyle{remark}
\newtheorem{remark}[proposition]{Remark}

\newcommand{\cF}{{\bf F}} %
\newcommand{\cX}{{\bf X}}%
\newcommand{\cY}{{\bf Y}}%
\newcommand{\cE}{{\bf E}}%
\newcommand{\cR}{{\bf R}}%
\newcommand{\cT}{{\bf T}}%

\hypersetup{
  pdfauthor={Michael Borinsky, Karen Vogtmann},
  pdftitle={Computing Euler characteristics using quantum field theory},
  pdfsubject={},
  pdfkeywords={},
  linkcolor  = RedViolet!85!black,
  citecolor  = BlueGreen!85!black,
  urlcolor   = YellowOrange!85!black,
  colorlinks = true,
}

\title[Computing Euler characteristics using quantum field theory]{Computing Euler characteristics using \\quantum field theory}
 \author{Michael Borinsky \and Karen Vogtmann}
\thanks{MB was supported by Dr.\ Max R\"ossler, the Walter Haefner Foundation and the ETH Z\"urich Foundation.}

\address{
Michael Borinsky\\
Institute for Theoretical Studies\\
ETH Z\"urich\\
8092 Z\"urich\\
Switzerland
}
\address{
Karen Vogtmann\\
University of Warwick\\
Mathematics Institute\\
Zeeman Building\\
Coventry CV4 7AL\\
United Kingdom
}

\begin{document}

\begin{abstract}
This paper explains how to use quantum field theory techniques to find formal power series that encode the virtual Euler characteristics of $\Out(F_n)$ and related graph complexes.  Finding such power series was a necessary step in the asymptotic analysis of $\chi(\Out(F_n))$ carried out in the authors' previous paper.
\end{abstract} 

\maketitle

\section{Introduction}

This article is based on the second author's talk at the conference Artin Groups, CAT(0) geometry and related topics (Charneyfest) in July 2021.   It is an exposition of some of the techniques that the authors used in their calculation of the  Euler characteristic of the group $\Out(F_n)$~\cite{BoVo}.  By results of Kontsevich, this  Euler characteristic is the same as the orbifold Euler characteristic of   the \emph{Lie graph complex} he defined in~\cite{Kont}.  Kontsevich also defined  \emph{commutative}  and \emph{associative} graph complexes, and discussed their orbifold Euler characteristics in Section~7.B of that  paper (Section~7.2 in a circulating preprint version). 

We will use the term {\em virtual} Euler characteristic as a uniform term for both Euler characteristics mentioned in the above paragraph.  It is different from the alternating sum of the Betti numbers, which we will call the \emph{classical} Euler characteristic.   The virtual Euler characteristic is used in group theory partly because it has better properties with respect to short exact sequences and partly because it has deep connections to number theory.  It is defined by taking the classical Euler characteristic of a  torsion-free  subgroup of finite index and dividing by the index. 
 Kontsevich defined the orbifold Euler characteristic of a graph complex on the chain level, by counting each generator  (i.e.\ graph $G$)  with weight  $1/|\Aut(G)|$.  

In this paper we use the action of $\Out(F_n)$ on Outer space  to reduce   various virtual Euler characteristic computations  to graph-counting problems, then use quantum field theory (QFT) methods to solve those problems.      This paper can serve as a  tutorial on how this works, and in particular is an explanation of Equations~(1) and (3) of~\cite{Kont}~Section~7.B. Such QFT methods were studied and further developed in the first author's thesis~\cite{Borinsky}.

As Kontsevich notes, the virtual Euler characteristic is much easier to compute than the classical Euler characteristic. However, the classical Euler characteristic is what one needs to infer information about  the dimension of the actual cohomology groups. 
To compute the classical Euler characteristic one must incorporate information about the automorphism groups of graphs into the count.
Kontsevich writes that it is reasonable to expect that `most'  graphs have no nontrivial automorphisms.   From that remark one might expect that the two Euler characteristics of $\Out(F_n)$ are asymptotically the same.   However,  in~\cite{BoVo2} we combine techniques described  here with asymptotic analysis to prove  that the ratio of the two is asymptotically equal to $e^{-\frac{1}{4}}$, so in particular the difference between the two grows as the rank $n$ grows.   In particular this  solves   Problem~6.5 of the paper~\cite{MSS} of Morita, Sakasai and Suzuiki.

\section{Counting admissible graphs}\label{intro}

To a topologist, a graph is a 1-dimensional CW-complex.  Finite graphs, all of whose vertices have valence at least three are called {\em admissible}.  Note that a leaf, to a topologist, is an edge terminating in a 1-valent vertex, so an admissible graph does not have leaves. One checks easily that there are only a finite number of admissible graphs with a given Euler characteristic.    We remind the reader that only connected admissible graphs occur in Outer space~\cite{CV}.

In this paper, however, we will use the following combinatorial definition of a graph, which is much more convenient  for solving counting problems.

\begin{definition} \label{def:graph}
A {\em graph} is a triple $G=(H,E,V)$, where $H$ is a finite set (the half-edges), $E$ is an involution (whose orbits are the edges and leaves of $G$), and $V$ is a partition of $H$ (the vertices of $G$).
(See Figure~\ref{HEV}.)  An  {\em  isomorphism} between two graphs is  a bijection between the respective sets of half-edges that
commutes with the involution and respects the partitions.
\end{definition} 

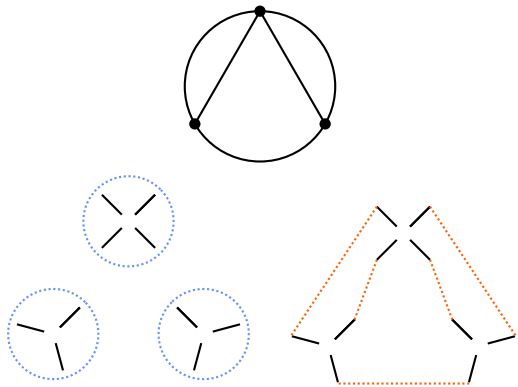
\begin{figure}
\begin{center}
 \begin{tikzpicture} \coordinate (t) at (0,1);\coordinate (l) at (210:1); \coordinate (r) at (-30:1); \draw[thick] (0,0) circle(1); \vertex{(t)}; \vertex{(l)}; \vertex{(r)}; \draw [thick] (l) to (t); \draw [thick] (r) to (t); \end{tikzpicture}
\end{center}

\begin{center}
\begin{tikzpicture}[scale=.5] \begin{scope}[rotate=45] \hecycle{4} \draw [col2, thick, densely dotted] (0,0) circle (1.2); \end{scope} \begin{scope}[xshift=-2cm, yshift=-3cm, rotate=45] \hecycle{3} \draw [col2,thick, densely dotted] (0,0) circle (1.2); \end{scope} \begin{scope}[xshift=2cm, yshift=-3cm, rotate=--135 ] \hecycle{3} \draw [col2, thick, densely dotted] (0,0) circle (1.2); \end{scope} \end{tikzpicture}
\,\,\,\,
\begin{tikzpicture}[scale=.5] \begin{scope}[rotate=45] \hecycle{4} \end{scope} \begin{scope}[xshift=-2cm, yshift=-3cm, rotate=45] \hecycle{3} \end{scope} \begin{scope}[xshift=2cm, yshift=-3cm, rotate=--135 ] \hecycle{3} \end{scope} \draw[col1,thick,densely dotted] (-45:1) to ( 1.3,-2.3); \draw[col1, thick,densely dotted] (-135:1) to ( -1.3,-2.3); \draw[col1, thick,densely dotted] (45:1) to ( 3,-2.75); \draw[col1, thick,densely dotted] (135:1) to ( -3,-2.75); \draw[col1, thick,densely dotted] (1.75,-4) to ( -1.75,-4); \end{tikzpicture}
\end{center}
\caption{
A graph $G$ with its representation as triple $(H,E,V)$ below.
Half-edges are depicted as short black lines, the partition $V$ is indicated with dotted blue circles and the matching $E$ of half-edges,
with dotted orange lines.}\label{HEV}
\end{figure}

Note that an isolated vertex is not considered to be a graph by this definition, but it does allow for the empty graph.  Also notice the difference between univalent vertices and leaves:  a univalent vertex is a block of the partition $V$ consisting of exactly one half-edge, whereas a leaf is a half-edge that is fixed by the involution $E$.  In this section and the next  we will assume each block of $V$ has at least 3 elements, i.e.\ the partition is {\em fat}; this corresponds to assuming all vertices of $G$ have valence at least 3.  In Sections~\ref{intro}--\ref{secgc2}, we will only consider graphs where the involution has no fixed points, so $E$ is a total matching of the half-edges; in other words we only consider admissible graphs 
for now.

\begin{remark}
Leaves of graphs are often called \emph{hairs} or \emph{legs}. 
The latter name is used in QFT and specifically in~\cite{Borinsky}. The name leaf is standard in more combinatorial contexts and was used in~\cite{BoVo}. Here, we will stick to this convention, as trees, connected graphs without cycles, are central to our considerations and trees carry leaves.
\end{remark}

We denote the set of automorphisms of a graph $G$ by $\Aut(G)$, and  
 by $v(G)$, $e(G)$ and $\chi(G) = v(G)-e(G)$   the number of vertices,  edges and the  Euler characteristic  of $G$.

Let $\mathcal G$ denote the set of (isomorphism classes of) admissible graphs, which we redefine to include  the empty graph,  and  $\mathcal G(m,k)$ the (finite) subset consisting of admissible graphs with $m$ edges and $k$ vertices. 
For instance, $\mathcal G(2,1) = \{ {\begin{tikzpicture}[x=1ex,y=1ex,baseline={([yshift=-.5ex]current bounding box.center)}] \coordinate (vm); \coordinate [left=.7 of vm] (v0); \coordinate [right=.7 of vm] (v1); \draw (v0) circle(.7); \draw (v1) circle(.7); \filldraw (vm) circle (1pt); \end{tikzpicture}} \}$ and $\mathcal G(3,2) = \{ {\begin{tikzpicture}[x=1ex,y=1ex,baseline={([yshift=-.5ex]current bounding box.center)}] \coordinate (v0); \coordinate [right=1.5 of v0] (v1); \coordinate [left=.7 of v0] (i0); \coordinate [right=.7 of v1] (o0); \draw (v0) -- (v1); \filldraw (v0) circle (1pt); \filldraw (v1) circle (1pt); \draw (i0) circle(.7); \draw (o0) circle(.7); \end{tikzpicture}}, {\begin{tikzpicture}[x=1ex,y=1ex,baseline={([yshift=-.5ex]current bounding box.center)}] \coordinate (vm); \coordinate [left=1 of vm] (v0); \coordinate [right=1 of vm] (v1); \draw (v0) -- (v1); \draw (vm) circle(1); \filldraw (v0) circle (1pt); \filldraw (v1) circle (1pt); \end{tikzpicture}}\}$.
Since  the set of admissible graphs with a given Euler characteristic is finite,
 the coefficient of $\hbar$ in the following formal power series is finite, i.e.\ this series serves as a generating function for (isomorphism classes of) admissible graphs.
\begin{align*} \cF(\hbar)=\sum_{G\in\mathcal G } \frac{\hbar^{-\chi(G)}}{|\Aut(G)|} &= 1 + \frac{\hbar}{|\Aut( {\begin{tikzpicture}[x=1ex,y=1ex,baseline={([yshift=-.5ex]current bounding box.center)}] \coordinate (vm); \coordinate [left=.7 of vm] (v0); \coordinate [right=.7 of vm] (v1); \draw (v0) circle(.7); \draw (v1) circle(.7); \filldraw (vm) circle (1pt); \end{tikzpicture}} )|} + \frac{\hbar}{|\Aut( {\begin{tikzpicture}[x=1ex,y=1ex,baseline={([yshift=-.5ex]current bounding box.center)}] \coordinate (v0); \coordinate [right=1.5 of v0] (v1); \coordinate [left=.7 of v0] (i0); \coordinate [right=.7 of v1] (o0); \draw (v0) -- (v1); \filldraw (v0) circle (1pt); \filldraw (v1) circle (1pt); \draw (i0) circle(.7); \draw (o0) circle(.7); \end{tikzpicture}} )|} + \frac{\hbar}{|\Aut( {\begin{tikzpicture}[x=1ex,y=1ex,baseline={([yshift=-.5ex]current bounding box.center)}] \coordinate (vm); \coordinate [left=1 of vm] (v0); \coordinate [right=1 of vm] (v1); \draw (v0) -- (v1); \draw (vm) circle(1); \filldraw (v0) circle (1pt); \filldraw (v1) circle (1pt); \end{tikzpicture}} )|} + \ldots \\
&=1 + \frac{\hbar}{8} + \frac{\hbar}{8} + \frac{\hbar}{12} + \ldots =1 + \frac{1}{3} \hbar + \frac{41}{36} \hbar^2 +\ldots \end{align*}
We use the symbol $\hbar$ as a formal variable that marks the
negative Euler characteristic to conform with typical quantum field theory notation.

All admissible graphs with Euler characteristic $-2$ contribute to the $\hbar^2$ coefficients in the expression above. These graphs are 
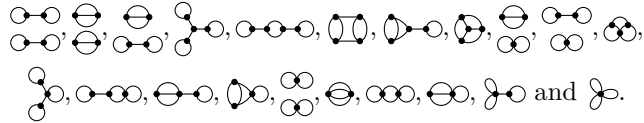
\begin{gather*} { \begin{tikzpicture}[x=1ex,y=1ex,baseline={([yshift=-.5ex]current bounding box.center)}] \coordinate (v); \coordinate [above=1.2 of v] (v01); \coordinate [right=1.5 of v01] (v11); \coordinate [left=.7 of v01] (i01); \coordinate [right=.7 of v11] (o01); \draw (v01) -- (v11); \filldraw (v01) circle (1pt); \filldraw (v11) circle (1pt); \draw (i01) circle(.7); \draw (o01) circle(.7); \coordinate [below=1.2 of v] (v02); \coordinate [right=1.5 of v02] (v12); \coordinate [left=.7 of v02] (i02); \coordinate [right=.7 of v12] (o02); \draw (v02) -- (v12); \filldraw (v02) circle (1pt); \filldraw (v12) circle (1pt); \draw (i02) circle(.7); \draw (o02) circle(.7); \end{tikzpicture} }, { \begin{tikzpicture}[x=1ex,y=1ex,baseline={([yshift=-.5ex]current bounding box.center)}] \coordinate (v); \coordinate [above=1.2 of v](v01); \coordinate [right=2 of v01] (v11); \coordinate [right=1 of v01] (vm1); \draw (v01) -- (v11); \draw (vm1) circle(1); \filldraw (v01) circle (1pt); \filldraw (v11) circle (1pt); \coordinate [below=1.2 of v](v02); \coordinate [right=2 of v02] (v12); \coordinate [right=1 of v02] (vm2); \draw (v02) -- (v12); \draw (vm2) circle(1); \filldraw (v02) circle (1pt); \filldraw (v12) circle (1pt); \end{tikzpicture} } , { \begin{tikzpicture}[x=1ex,y=1ex,baseline={([yshift=-.5ex]current bounding box.center)}] \coordinate (v); \coordinate [above=1.2 of v] (vm1); \coordinate [left=1 of vm1] (v01); \coordinate [right=1 of vm1] (v11); \draw (v01) -- (v11); \draw (vm1) circle(1); \filldraw (v01) circle (1pt); \filldraw (v11) circle (1pt); \coordinate [below=1.2 of v] (vm2); \coordinate [left=.75 of vm2] (v02); \coordinate [right=.75 of vm2] (v12); \coordinate [left=.7 of v02] (i02); \coordinate [right=.7 of v12] (o02); \draw (v02) -- (v12); \filldraw (v02) circle (1pt); \filldraw (v12) circle (1pt); \draw (i02) circle(.7); \draw (o02) circle(.7); \end{tikzpicture} } , { \begin{tikzpicture}[x=1ex,y=1ex,baseline={([yshift=-.5ex]current bounding box.center)}] \coordinate (v) ; \def \n {3} \def \rad {1.2} \def \rud {1.9} \foreach \s in {1,...,\n} { \def \angle {360/\n*(\s - 1)} \coordinate (s) at ([shift=({\angle}:\rad)]v); \coordinate (u) at ([shift=({\angle}:\rud)]v); \draw (v) -- (s); \filldraw (s) circle (1pt); \draw (u) circle (.7); } \filldraw (v) circle (1pt); \end{tikzpicture} } , { \begin{tikzpicture}[x=1ex,y=1ex,baseline={([yshift=-.5ex]current bounding box.center)}] \coordinate (v) ; \def \n {2} \def \rad {.7} \def \rud {2} \def \rid {2.7} \foreach \s in {1,...,\n} { \def \angle {360/\n*(\s - 1)} \coordinate (s) at ([shift=({\angle}:\rad)]v); \coordinate (u) at ([shift=({\angle}:\rud)]v); \coordinate (t) at ([shift=({\angle}:\rid)]v); \draw (s) -- (u); \filldraw (u) circle (1pt); \filldraw (s) circle (1pt); \draw (t) circle (.7); } \draw (v) circle(\rad); \end{tikzpicture} } , { \begin{tikzpicture}[x=1ex,y=1ex,baseline={([yshift=-.5ex]current bounding box.center)}] \coordinate (v0); \coordinate[above left=1.5 of v0] (v1); \coordinate[below left=1.5 of v0] (v2); \coordinate[above right=1.5 of v0] (v3); \coordinate[below right=1.5 of v0] (v4); \draw (v1) to[bend left=80] (v2); \draw (v1) to[bend right=80] (v2); \draw (v3) to[bend right=80] (v4); \draw (v3) to[bend left=80] (v4); \draw (v1) -- (v3); \draw (v2) -- (v4); \filldraw (v1) circle(1pt); \filldraw (v2) circle(1pt); \filldraw (v3) circle(1pt); \filldraw (v4) circle(1pt); \end{tikzpicture} } , { \begin{tikzpicture}[x=1ex,y=1ex,baseline={([yshift=-.5ex]current bounding box.center)}] \coordinate (v0); \coordinate[right=.5 of v0] (v3); \coordinate[right=1.5 of v3] (v4); \coordinate[above left=1.5 of v0] (v1); \coordinate[below left=1.5 of v0] (v2); \coordinate[right=.7 of v4] (o); \draw (v3) to[bend left=20] (v2); \draw (v3) to[bend right=20] (v1); \draw (v1) to[bend right=80] (v2); \draw (v1) to[bend left=80] (v2); \draw (v3) -- (v4); \filldraw (v1) circle(1pt); \filldraw (v2) circle(1pt); \filldraw (v3) circle(1pt); \filldraw (v4) circle(1pt); \draw (o) circle(.7); \end{tikzpicture} } , { \begin{tikzpicture}[x=1ex,y=1ex,baseline={([yshift=-.5ex]current bounding box.center)}] \coordinate (v); \def \n {3} \def \rad {1.2} \foreach \s in {1,...,\n} { \def \angle {360/\n*(\s - 1)} \coordinate (s) at ([shift=({\angle}:\rad)]v); \draw (v) -- (s); \filldraw (s) circle (1pt); } \draw (v) circle (\rad); \filldraw (v) circle (1pt); \end{tikzpicture} } , { \begin{tikzpicture}[x=1ex,y=1ex,baseline={([yshift=-.5ex]current bounding box.center)}] \coordinate (v); \coordinate [above=1.2 of v](vm1); \coordinate [left=1 of vm1] (v01); \coordinate [right=1 of vm1] (v11); \draw (v01) -- (v11); \draw (vm1) circle(1); \filldraw (v01) circle (1pt); \filldraw (v11) circle (1pt); \coordinate [below=1.2 of v] (vm2); \coordinate [left=.7 of vm2] (v02); \coordinate [right=.7 of vm2] (v12); \draw (v02) circle(.7); \draw (v12) circle(.7); \filldraw (vm2) circle (1pt); \end{tikzpicture} } , { \begin{tikzpicture}[x=1ex,y=1ex,baseline={([yshift=-.5ex]current bounding box.center)}] \coordinate (v); \coordinate [above=1.2 of v] (vm1); \coordinate [left=.75 of vm1] (v01); \coordinate [right=.75 of vm1] (v11); \coordinate [left=.7 of v01] (i01); \coordinate [right=.7 of v11] (o01); \draw (v01) -- (v11); \filldraw (v01) circle (1pt); \filldraw (v11) circle (1pt); \draw (i01) circle(.7); \draw (o01) circle(.7); \coordinate [below=1.2 of v] (vm2); \coordinate [left=.7 of vm2] (v02); \coordinate [right=.7 of vm2] (v12); \draw (v02) circle(.7); \draw (v12) circle(.7); \filldraw (vm2) circle (1pt); \end{tikzpicture} } , { \begin{tikzpicture}[x=1ex,y=1ex,baseline={([yshift=-.5ex]current bounding box.center)}] \coordinate (vm); \coordinate [left=.7 of vm] (v0); \coordinate [right=.7 of vm] (v1); \coordinate [above=.7 of v0] (v2); \coordinate [above=.7 of v1] (v3); \draw (v0) circle(.7); \draw (v1) circle(.7); \draw (v3) arc(0:180:.7) (v2); \filldraw (vm) circle (1pt); \filldraw (v2) circle (1pt); \filldraw (v3) circle (1pt); \end{tikzpicture} } ,\\
{ \begin{tikzpicture}[x=1ex,y=1ex,baseline={([yshift=-.5ex]current bounding box.center)}] \coordinate (v) ; \def \n {3} \def \rad {1.2} \def \rud {1.9} \foreach \s in {2,...,\n} { \def \angle {360/\n*(\s - 1)} \coordinate (s) at ([shift=({\angle}:\rad)]v); \coordinate (u) at ([shift=({\angle}:\rud)]v); \draw (v) -- (s); \filldraw (s) circle (1pt); \draw (u) circle (.7); } \filldraw (v) circle (1pt); \coordinate (s) at ([shift=({0}:.7)]v); \draw (s) circle (.7); \end{tikzpicture} } , { \begin{tikzpicture}[x=1ex,y=1ex,baseline={([yshift=-.5ex]current bounding box.center)}] \coordinate (v0); \coordinate [right=1.5 of v0] (v1); \coordinate [left=.7 of v0] (i0); \coordinate [right=.7 of v1] (o0); \coordinate [right=.7 of o0] (v2); \coordinate [right=.7 of v2] (o1); \draw (v0) -- (v1); \filldraw (v0) circle (1pt); \filldraw (v1) circle (1pt); \filldraw (v2) circle (1pt); \draw (i0) circle(.7); \draw (o0) circle(.7); \draw (o1) circle(.7); \end{tikzpicture} } , { \begin{tikzpicture}[x=1ex,y=1ex,baseline={([yshift=-.5ex]current bounding box.center)}] \coordinate (vm); \coordinate [left=1 of vm] (v0); \coordinate [right=1 of vm] (v1); \coordinate [right=1.5 of v1] (v2); \coordinate [right=.7 of v2] (o); \draw (v0) -- (v1); \draw (v1) -- (v2); \draw (vm) circle(1); \draw (o) circle(.7); \filldraw (v0) circle (1pt); \filldraw (v1) circle (1pt); \filldraw (v2) circle (1pt); \end{tikzpicture} } , { \begin{tikzpicture}[x=1ex,y=1ex,baseline={([yshift=-.5ex]current bounding box.center)}] \coordinate (v0); \coordinate[right=.5 of v0] (v3); \coordinate[above left=1.5 of v0] (v1); \coordinate[below left=1.5 of v0] (v2); \coordinate[right=.7 of v3] (o); \draw (v3) to[bend left=20] (v2); \draw (v3) to[bend right=20] (v1); \draw (v1) to[bend right=80] (v2); \draw (v1) to[bend left=80] (v2); \filldraw (v1) circle(1pt); \filldraw (v2) circle(1pt); \filldraw (v3) circle(1pt); \draw (o) circle(.7); \end{tikzpicture} } , { \begin{tikzpicture}[x=1ex,y=1ex,baseline={([yshift=-.5ex]current bounding box.center)}] \coordinate (v); \coordinate [above=1.2 of v] (vm1); \coordinate [left=.7 of vm1] (v01); \coordinate [right=.7 of vm1] (v11); \draw (v01) circle(.7); \draw (v11) circle(.7); \filldraw (vm1) circle (1pt); \coordinate [below=1.2 of v] (vm2); \coordinate [left=.7 of vm2] (v02); \coordinate [right=.7 of vm2] (v12); \draw (v02) circle(.7); \draw (v12) circle(.7); \filldraw (vm2) circle (1pt); \end{tikzpicture} } , { \begin{tikzpicture}[x=1ex,y=1ex,baseline={([yshift=-.5ex]current bounding box.center)}] \coordinate (vm); \coordinate [left=1 of vm] (v0); \coordinate [right=1 of vm] (v1); \draw (v0) to[bend left=45] (v1); \draw (v0) to[bend right=45] (v1); \draw (vm) circle(1); \filldraw (v0) circle (1pt); \filldraw (v1) circle (1pt); \end{tikzpicture} } , { \begin{tikzpicture}[x=1ex,y=1ex,baseline={([yshift=-.5ex]current bounding box.center)}] \coordinate (vm); \coordinate [left=.7 of vm] (v0); \coordinate [right=.7 of vm] (v1); \coordinate [left=.7 of v0] (vc1); \coordinate [right=.7 of v1] (vc2); \draw (vc1) circle(.7); \draw (vc2) circle(.7); \draw (vm) circle(.7); \filldraw (v0) circle (1pt); \filldraw (v1) circle (1pt); \end{tikzpicture} } , { \begin{tikzpicture}[x=1ex,y=1ex,baseline={([yshift=-.5ex]current bounding box.center)}] \coordinate (vm); \coordinate [left=1 of vm] (v0); \coordinate [right=1 of vm] (v1); \coordinate [right=.7 of v1] (o); \draw (v0) -- (v1); \draw (vm) circle(1); \draw (o) circle(.7); \filldraw (v0) circle (1pt); \filldraw (v1) circle (1pt); \end{tikzpicture} } , { \begin{tikzpicture}[x=1ex,y=1ex,baseline={([yshift=-.5ex]current bounding box.center)}] \coordinate (v) ; \def \rad {1.5} \coordinate (s1) at ([shift=(0:1.2)]v); \coordinate (s2) at ([shift=(120:\rad)]v); \coordinate (s3) at ([shift=(240:\rad)]v); \coordinate [right=.7 of s1] (o); \draw (v) to[out=180,in=210] (s2) to[out=30,in=60] (v); \draw (v) to[out=300,in=330] (s3) to[out=150,in=180] (v); \draw (v) -- (s1); \filldraw (v) circle (1pt); \filldraw (s1) circle (1pt); \draw (o) circle(.7); \end{tikzpicture} } \text{ and } { \begin{tikzpicture}[x=1ex,y=1ex,baseline={([yshift=-.5ex]current bounding box.center)}] \coordinate (v); \def \rad {1.5} \coordinate (s1) at ([shift=(0:\rad)]v); \coordinate (s2) at ([shift=(120:\rad)]v); \coordinate (s3) at ([shift=(240:\rad)]v); \draw (v) to[out=60,in=90] (s1) to[out=-90,in=0-60] (v); \draw (v) to[out=180,in=210] (s2) to[out=30,in=60] (v); \draw (v) to[out=300,in=330] (s3) to[out=150,in=180] (v); \filldraw (v) circle (1pt); \end{tikzpicture} }. \end{gather*}
This illustrates that the number of graphs becomes 
quickly too large to be handled one-by-one.

\subsection{Labeled graphs versus isomorphism classes of graphs}\label{unlabeled} Although we want to count isomorphism classes of  graphs, it is convenient to start with {\em labeled}  graphs, where the half-edges are the set of consecutive integers $H = \{1,2,\ldots,|H|\}$ (see   Figure~\ref{labels} for an example).  %
A labeled graph will be denoted $lG$.
The set of labeled admissible graphs with $m$  edges and  $k$ vertices  will be denoted  $\mathcal{LG}(m,k)$.

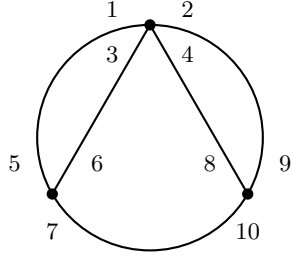
\begin{figure}
\begin{center}
 \begin{tikzpicture} \coordinate (t) at (0,1.5);\coordinate (l) at (210:1.5); \coordinate (r) at (-30:1.5); \draw[thick] (0,0) circle(1.5); \vertex{(t)}; \vertex{(l)}; \vertex{(r)}; \draw [thick] (l) to (t); \draw [thick] (r) to (t); \node (1) at ($(t)+(-.5,.2)$) {\small{$1$}}; \node (2) at ($(t)+(.5,.2)$) {\small{$2$}}; \node (3) at ($(t)+(-.5,-.4)$) {\small{$3$}}; \node (4) at ($(t)+(.5,-.4)$) {\small{$4$}}; \node (5) at ($(l)+(-.5,.4)$) {\small{$5$}}; \node (6) at ($(l)+(.6,.4)$) {\small{$6$}}; \node (7) at ($(l)+(0,-.5)$) {\small{$7$}}; \node (8) at ($(r)+(-.5,.4)$) {\small{$8$}}; \node (9) at ($(r)+(.5,.4)$) {\small{$9$}}; \node (10) at ($(r)+(0,-.5)$) {\small{$10$}}; \end{tikzpicture}
\end{center}
\caption{A half-edge labeled graph $lG$.}\label{labels}
\end{figure}

Two labeled graphs are equal if and only if they are represented by identical triples $(H,E,V)$.
The following proposition relates the count of labeled graphs with the number of isomorphism classes of graphs.

\begin{proposition}
\label{prop:graph_sum}
$$ \sum_{G\in\mathcal G(m,k) } \frac{1}{|\Aut(G)|}=\frac{|\mathcal{LG}(m,k)|}{(2m)!}.$$
\end{proposition}

\begin{proof}
The symmetric group $\Sigma_{2m}$ acts on $\mathcal{LG}(m,k)$ by permuting the labels.  The orbit-stabilizer theorem  says
$$(2m)!=  |(\hbox{orbit }lG)| \cdot |(\hbox{stabilizer }lG)|.$$ The orbit of a labeled graph $lG$ is nothing but an isomorphism class
of labeled graphs i.e.\ an unlabeled graph, and the stabilizer of $lG$ is isomorphic to $\Aut(G)$, giving
\begin{gather*} |\mathcal{LG}(m,k)|= \sum_{\rm orbits}|(\rm orbit)|=\sum_{G\in \mathcal{G}(m,k)} \frac{(2m)!}{|\Aut(G)|}. \qedhere \end{gather*}
\end{proof}
Multiplying by $\hbar^{m-k}$ and summing over all $m$ and $k$ gives the following expression for the generating function $\cF$:
$$ \cF(\hbar)=\sum_{G\in\mathcal G } \frac{\hbar^{-\chi(G)}}{|\Aut(G)|}=\sum_{m,k \geq 0}\frac{1}{(2m)!}|\mathcal{LG}(m,k)|\hbar^{m-k}.$$

\subsection{Counting labeled graphs}\label{labeled} Now we want to count labeled graphs.  Given $2m$ half-edges, any matching of $H$ and any fat partition $V$ of $H$ corresponds to a unique labeled admissible graph, so to count the number of such graphs we can make with $2m$ half-edges, we need to count the number of matchings and the number of fat partitions.   These counts are standard, we record them below.

\begin{lemma}\label{lem:match} There are $(2m-1)!! = (2m-1) \cdot (2m-3) \cdots 3 \cdot 1$ matchings of a set of $2m$ labeled elements.
\end{lemma}
\begin{proof} This is an easy induction on $m$.
\end{proof}

\begin{lemma}\label{lem:fat} The coefficient of $\lambda^k\frac{x^n}{n!}$ in $\exp(\lambda(e^x-1-x-\frac{x^2}{2}))$ is the number of ways of partitioning the set $\{1,\ldots,n\}$ into $k$ sets, each of size at least $3$.
\end{lemma}
 \begin{proof}The number of partitions of a set of $n$ elements into non-empty pieces is the coefficient of $\frac{x^{n}}{n!}$ in the power series expansion of $\exp(e^x-1)$:
\begin{align*} \exp(e^x-1) &= \exp\left(x+\frac{x^2}{2!}+\frac{x^3}{3!}+\ldots\right)\\
 &=1+\sum_{i\geq 1} \frac{x^i}{i!} + \frac{1}{2!} \sum_{i,j\geq 1} \frac{x^i}{i!} \frac{x^j}{j!}+\frac{1}{3!}\sum_{i,j,k \geq 1} \frac{x^i}{i!} \frac{x^j}{j!}\frac{x^k}{k!}+\ldots \end{align*}
If instead we take $\exp(\lambda(e^x-1)),$ the power of $\lambda$ keeps track of the number of pieces of   the partition.
\begin{align*} \exp(\lambda(e^x-1)) &= \exp\left(\lambda\left(x+\frac{x^2}{2!}+\frac{x^3}{3!}+\ldots\right)\right)\\
 &=1+\lambda\sum_{i\geq 1} \frac{x^i}{i!} + \frac{\lambda^2}{2!}\sum_{i,j \geq 1} \frac{x^i}{i!} \frac{x^j}{j!}+\frac{\lambda^3}{3!}\sum_{i,j,k \geq 1} \frac{x^i}{i!} \frac{x^j}{j!}\frac{x^k}{k!}+\ldots \end{align*}
Replacing $e^x-1$ by  $e^x-1-x-\frac{x^2}{2}$ in the above formulas restricts attention to fat partitions, i.e.\ partitions with at least three elements in each block.
\end{proof}

Since an admissible labeled graph  consists of
a fat partition of the set $\{1,\ldots,|H|\}$ and
a matching on $H$, Lemmas~\ref{lem:match} and ~\ref{lem:fat} give
\begin{gather*} |\mathcal{LG}(m,k)|= (2m-1)!!\big[\lambda^k\frac{x^{2m}}{(2m)!}\big]\exp\left(\lambda\left(e^x-1-x-\frac{x^2}{2}\right)\right) \\
=(2m-1)!!(2m)![\lambda^k x^{2m}]\exp\left(\lambda\left(e^x-1-x-\frac{x^2}{2}\right)\right). \end{gather*}
Here the square bracket denotes the {\em coefficient extraction operator}, which extracts the coefficient of the bracketed monomial expression from the following power series.  If $\cF$ is a formal power series in commuting variables $x,y,z,\ldots$  and $M$ is a monomial in these variables, the notation $[M]\cF$ denotes the coefficient of $M$ in $\cF$; this is itself  a power series in the   variables that do not occur in $M$.

Since you can only make a finite number of labeled graphs with $2m$ labeled half-edges, the function $$p_m(\lambda)=(2m-1)!![x^{2m}]\exp\left(\lambda\left(e^x-1-x-\frac{x^2}{2}\right)\right)$$ is a polynomial in $\lambda$,
$$
p_m(\lambda)=\frac{|\mathcal{LG}(m,1)|}{(2m)!}\lambda + \frac{|\mathcal{LG}(m,2)|}{(2m)!}\lambda^2+\ldots +\frac{|\mathcal{LG}(m,v_{\text{max}})|}{(2m)!}\lambda^{v_{\text{max}}},
$$
where ${v_{\text{max}}} = \deg p_m \leq \frac23 m$, as each admissible graph with $2m$ half-edges has at most $\frac23m$ vertices.
Using this with Proposition~\ref{prop:graph_sum} gives the following formula for our generating function:
\[
 \cF(\hbar)= \sum_{G \in \mathcal G} \frac{\hbar^{-\chi(G)}}{|\Aut(G)|}=\sum_{m,k \geq 0}\frac{|\mathcal{LG}(m,k)|}{(2m)!}\hbar^{m-k}
 =\sum_{m\geq 0} \hbar^m p_m(\hbar^{-1})
\]
i.e.
\begin{equation}\label{Gamma}
 \cF(\hbar) =\sum_{m\geq 0} \hbar^{m} (2m-1)!![x^{2m}]\exp\left(\hbar^{-1}\left(e^x-1-x-\frac{x^2}{2}\right)\right).
\end{equation}

\subsection{Connected graphs}\label{Ez}
When we counted how many admissible graphs we could make from $2m$ labeled edges, we did not care whether the results were connected or disconnected.  If we only want to count connected graphs, we can take the formal logarithm of the generating function. This is standard in the combinatorics literature; see, e.g., Stanley's classic book \cite{Stanley}, Chapter~5.1.

The basic principle is that, given an exponential  generating function that counts some type of (labeled) objects, formally exponentiating this series counts collections of such objects.  We have already used this to count graphs:  the generating function $\sum \frac{ x^n}{n!}$ counts the number of labeled  \emph{corollas} made of one vertex and $n$ half-edges, weighted by the automorphism group of the corolla (the symmetric group $\Sigma_n$).  An admissible graph   is then just a collection of corollas (with at least 3 half-edges each) together with a matching of half-edges.  To count collections of corollas, we exponentiated $\lambda\sum_{{n\geq 3}} \frac{ x^n}{n!}$,  then multiplied by the number of matchings to get the number of graphs.

\subsection{Gaussian integrals  and 0-dimensional QFT}
\label{gauss}

We can replace the expression $\hbar^m(2m-1)!!$ in   Formula~(\ref{Gamma}) by a Gaussian integral.  We have
$$
\frac{1}{\sqrt{2\pi \hbar}}
\int_{-\infty}^\infty e^{-\frac{x^2}{2\hbar}}x^ndx =
\begin{cases}
0  & n \text{ odd}\\
(2m-1)!!\hbar^m &n=2m
\end{cases}
$$
Using this we can  ignore the fact that we are only interested in extracting even powers of $x$ in
$\exp(\hbar^{-1}(e^x-1-x-\frac{x^2}{2}))$, i.e.\ we can write
$$
 \cF(\hbar)=\sum_{n\geq 3} \left(\frac{1}{\sqrt{2\pi \hbar}}\int_{-\infty}^\infty e^{-\frac{x^2}{2\hbar}}x^ndx\right)[x^{n}]\exp\left(\hbar^{-1}\left(e^x-1-x-\frac{x^2}{2}\right)\right)
$$
We now do something that would be unforgivable if we were dealing with power series representatives of actual functions, namely we interchange the summation and integral to get
\begin{align*} \cF(\hbar)\,``&="\,\frac{1}{\sqrt{2\pi \hbar}}\int e^{-\frac{x^2}{2\hbar}}\exp\left(\hbar^{-1}\left(e^x-1-x-\frac{x^2}{2}\right)\right)\,dx\\
&=\frac{1}{\sqrt{2\pi \hbar}}\int \exp\left(\hbar^{-1}\left(e^x-1-x-x^2 \right)\right)\,dx \end{align*}
The series from which we derived this integral  is formally understood as the \emph{path-integral} for the $0$-dimensional scalar quantum field theory with  potential   $e^x-1-x-\frac{x^2}{2}$. %
The generating function $\cF(\hbar)$ is the \emph{partition function} and $\log \cF(\hbar)$, which counts the associated set of connected graphs, is the \emph{free energy} associated to this scalar QFT.
For more information about this perspective, see \cite{Borinsky} Chapter~3.

\section{The virtual Euler characteristic of the even commutative graph complex}
\label{secgc2}
The even commutative graph complex $\mathcal{GC}_2$ breaks into subcomplexes $\mathcal{GC}_2^{(n)}$ generated by   connected admissible graphs with   fundamental group $F_n$, modulo the relation that an automorphism inducing an odd permutation of the edges changes the sign of the generator.  In particular, if there is an odd permutation the generator is zero. Sometimes $\mathcal{GC}_2$ is defined to exclude graphs with loops.  We will allow graphs with loops, but this makes no difference in the homology (or the Euler characteristic) (see \cite{Willwacher}, Proposition 3.4).

We recall the interpretation of $\mathcal{GC}_2^{(n)}$ in terms of Outer space (see \cite{CoVo}, Section 5).    Let $CV_n^*$ denote the simplicial closure of Outer space $CV_n$, and $CV_n^\infty=CV_n^*\setminus CV_n$ the subcomplex at infinity.  $\Out(F_n)$ acts on $CV_n^*$ preserving  $CV_n^\infty$, and the relative chain complex $$C_*\left(CV_n^*/\Out(F_n),CV_n^\infty/\Out(F_n);\Q\right)$$ of the quotient spaces has one generator in dimension $m$ for each connected admissible graph $G$ with $m+1$ edges that has no automorphisms inducing an odd permutation of its edges (see, e.g., \cite{HV}). In other words, this relative chain complex is precisely  $\mathcal{GC}_2^{(n)}$.

Let $\chi(\mathcal{GC}_2^{(n)})$ denote the orbifold Euler characteristic of this relative quotient (see, e.g., Brown~\cite{Brown}, Chapter IX, Section 7). Then   %
$$
\chi(\mathcal{GC}_2^{(n)})=\sum_{\langle \sigma \rangle} \frac{(-1)^{\dim \sigma}}{|\Stab(\sigma)|},
$$
where $\sigma$ runs over all orbit representatives of the action of $\Out(F_n)$ on $CV_n=CV_n^*\setminus CV_n^\infty$, %
i.e.\ all  connected admissible graphs with fundamental group $F_{n}$.  Assembling the
$\chi(\mathcal{GC}_2^{(n)})$ into a generating function (with a shift  $n\mapsto (n+1)$ to account for the fact that $n=1-\chi(G)$) gives
\begin{align*} \cX(\hbar)&=\sum_{n\geq 1} \chi(\mathcal{GC}_2^{(n+1)})\hbar^n=\sum_{G\in\mathcal G_c} {(-1)}^{e(G)}\frac{\hbar^{-\chi(G)}}{|\Aut(G)|} \\
&= \hbar \left( \frac{(-1)^2}{|\Aut( {\begin{tikzpicture}[x=1ex,y=1ex,baseline={([yshift=-.5ex]current bounding box.center)}] \coordinate (vm); \coordinate [left=.7 of vm] (v0); \coordinate [right=.7 of vm] (v1); \draw (v0) circle(.7); \draw (v1) circle(.7); \filldraw (vm) circle (1pt); \end{tikzpicture}} )|} + \frac{(-1)^3}{|\Aut( {\begin{tikzpicture}[x=1ex,y=1ex,baseline={([yshift=-.5ex]current bounding box.center)}] \coordinate (v0); \coordinate [right=1.5 of v0] (v1); \coordinate [left=.7 of v0] (i0); \coordinate [right=.7 of v1] (o0); \draw (v0) -- (v1); \filldraw (v0) circle (1pt); \filldraw (v1) circle (1pt); \draw (i0) circle(.7); \draw (o0) circle(.7); \end{tikzpicture}} )|} + \frac{(-1)^3}{|\Aut( {\begin{tikzpicture}[x=1ex,y=1ex,baseline={([yshift=-.5ex]current bounding box.center)}] \coordinate (vm); \coordinate [left=1 of vm] (v0); \coordinate [right=1 of vm] (v1); \draw (v0) -- (v1); \draw (vm) circle(1); \filldraw (v0) circle (1pt); \filldraw (v1) circle (1pt); \end{tikzpicture}} )|} \right) +\ldots = -\frac{1}{12} \hbar + \frac{1}{360} \hbar^3 + \ldots \end{align*}
where we sum over the set $\mathcal G_c$ of all  connected  graphs in $\mathcal G$.

Exponentiating $\cX(\hbar)$, invoking Proposition~\ref{prop:graph_sum} and recalling the definition of $p_m(\lambda)$ from Section~\ref{labeled} gives
\begin{align*} \exp(\cX(\hbar))=\cE(\hbar)&= \sum_{G\in\mathcal G } (-1)^{e(g)}\frac{\hbar^{-\chi(G)}}{|\Aut(G)|}=\sum_{m,k \geq 0}(-1)^m \frac{|\mathcal{LG}(m,k)|}{(2m)!}\hbar^{m-k}\\
&=\sum_{m\geq 0} (-1)^m \hbar^m p_m(\hbar^{-1})\\
&=\sum_{m\geq 0} (-\hbar)^{m} (2m-1)!![x^{2m}]\exp\left(\hbar^{-1}\left(e^x-1-x-\frac{x^2}{2}\right)\right) \end{align*}
We are interested in  $[\hbar^n]\cE(\hbar)$.  It is just as good to find a formula for $\cE(-\hbar)$, since $[\hbar^n]\cE(-\hbar)=(-1)^n [\hbar^n]\cE(\hbar)$. Now note
$$\cE(-\hbar)= \sum_{m\geq 0} \hbar^{m} (2m-1)!![x^{2m}]\exp\left(-\hbar^{-1}\left(e^x-1-x-\frac{x^2}{2}\right)\right).$$
If we again replace the factor $\hbar^m(2m-1)!!$ with a Gaussian integral and
naively interchange summation and integration as in Section~\ref{gauss}, we get  the path-integral of the $0$-dimensional   QFT with potential $-(e^x-1-x-\frac{x^2}{2})$ and action $-(e^x-1-x),$ written
\begin{align} \label{Epathintegral} \cE(-\hbar)=\frac{1}{\sqrt{2\pi \hbar}}\int e^{-\hbar^{-1}(e^x-1-x)}dx.\end{align}
Since $e^x-1-x = \sum_{n \geq 2} \frac{x^n}{n!}$ we recover Kontsevich's Equation~(3) from \cite{Kont}.

We can call the path-integral above a \emph{topological quantum field theory},
because the partition function $\cE(\hbar)$ and the free energy $\log \cE(\hbar) = \cX(\hbar)$
encode the topological invariants $\chi(\mathcal{GC}_2^{(n)})$. 

\subsection{Solving the path-integral}
For general potentials  such a  path-integral expression  does not converge, but in the case of \eqref{Epathintegral} it actually does as long as $\hbar$ is positive.
To see this, recall Euler's integral formula for the $\Gamma$-function:
$\Gamma(N)=\int_0^\infty t^{N-1}e^{-t}dt$.  The change of variables $t\mapsto Ne^x$ gives
$$\Gamma(N)=  \int_{-\infty}^\infty N^N e^{Nx-Ne^x}dx.$$
This looks tantalizingly close to the formula for $ \cE(-\frac{1}{N})$;   if we now divide by $N^Ne^{-N}\sqrt{2\pi\frac{1}{N}} $ we in fact get the exact formula: $$\frac{\Gamma(N)}{N^Ne^{-N}\sqrt{2\pi\frac{1}{N}}}= \frac{1}{\sqrt{2\pi\frac{1}{N}}}\int_{-\infty}^\infty e^{N(1+x-e^x)}dx=  \cE(-\tfrac{1}{N}).$$
Taking the log of both sides   gives
$$\log\Gamma(N)-N \ln N + N -\frac{1}{2}(\log 2\pi - \log N) = \log \cE(-\tfrac{1}{N})=\cX(-\tfrac{1}{N})=\sum_{G\in \mathcal G_c}\frac{(-1)^{e(G)}}{|\Aut(G)|}\left(-\frac{1}{N}\right)^{-\chi(G)}$$
We now recall that Stirling's asymptotic expansion of $\log\Gamma(N)$ for large $N$ in terms of Bernoulli numbers, which are themselves defined by the generating function $\sum_{n= 0}^\infty \frac{B_n}{n!} x^n = x/(e^x-1)$, (see, e.g.,  \cite{Bruijn}, \S3.10) is %
$$\log \Gamma(N) \sim N\ln N - N +  \frac{1}{2}(\log 2\pi  -  \log N) +\sum_{k=1}^{\infty} \frac {B_{2k}}{2k(2k-1)}{\frac{1}{N^{2k-1}}} \text{ for } N \rightarrow \infty,$$
so that the coefficient of  $\tfrac{1}{N^{n}}$ in $\cX(-\tfrac{1}{N})$ is equal to  $\frac {B_{2k}}{2k(2k-1)}$ if $n=2k-1$ and vanishes otherwise. As all even coefficients of $\cX(\hbar)$ therefore vanish we have $\cX(-\hbar) = -\cX(\hbar)$.

From a QFT perspective, we can say that 
we effectively \emph{solved} the QFT encoded in~\eqref{Epathintegral}
by finding a closed-form expression for the expansion coefficients 
of the free energy $\cX(\hbar)$.

We can summarize all our findings of this section in the following formula,
$$
\chi(\mathcal{GC}_2^{(n+1)}) =
[\hbar^n]
\cX(\hbar)=
- \sum_{\substack{G\in \mathcal G_c \\  \chi(G)=-n } } \frac{(-1)^{e(G)}}{|\Aut(G)|}=
\begin{dcases}
0 & \text{ if $n$ is even}\\
 -\frac {B_{n+1}}{n(n+1)}  &\text { if $n$ is odd}.
 \end{dcases}
$$

The formal interchange of summation and integration while passing to an integral representation can be justified by employing the formalism of asymptotic expansions and Laplace's method as   was done in~\cite[Proposition 5.3]{BoVo}.

\section{The virtual Euler characteristic of \texorpdfstring{$\Out(F_n)$}{OutFn}}

To compute the virtual Euler characteristic of $\Out(F_n)$ one uses the spine $K_n$ of Outer space instead of the whole space, since the action on the spine is cocompact.  The spine has the structure of a cube complex with one cube for every triple $(g,G,F)$ consisting of a connected admissible graph $G$, a subforest $F$ and an isomorphism $g\colon F_n\cong\pi_1(G)$. The action of $\Out(F_n)$ just changes the isomorphism $g$, so computing the orbifold Euler characteristic reduces to counting  pairs $(G,F)$, weighted by their stabilizers.  In this section we show how to do this, using the fact that the stabilizer of  the cube $(G,F)$ is isomorphic to $\Aut(G,F)$, the automorphisms of $G$ that send $F$ to itself.   By Kontsevich's theorem the virtual Euler characteristic of $\Out(F_n)$ is also the  virtual  Euler characteristic of the Lie graph complex~\cite{Kont}.

A pair $(G,F)$ consisting of an admissible graph and a spanning forest $F$ can be thought of as a collection of trees together with a matching of their leaves.  To count such pairs we should count trees, exponentiate to count forests, connect the leaves of forests in pairs, then take logarithm to restrict to connected graphs.

The dimension of the cube  in the spine of Outer space corresponding to $(G,F)$ is equal to the number of edges in the forest, so if we are interested in computing  Euler characteristic we only need to know the parity of $e(F)$.
Using Brown's formula~\cite{Brown} for the virtual Euler characteristic of a discrete group acting properly and cocompactly on a contractible cell complex, we therefore have
$$\chi(\Out(F_n)) = \sum_{(G,F)}\frac{(-1)^{e(F)}}{|\Aut(G,F)|},$$
where the sum is over isomorphism classes of pairs $(G,F)$ with $G$ a connected admissible graph and $\chi(G)=1-n$.

As in the previous sections it is convenient to define the generating function
\begin{align*} \cY(\hbar) &= \sum_{n \geq 1} \chi(\Out(F_{n+1})) \hbar^{n} = \sum_{(G,F)}\frac{(-1)^{e(F)}}{|\Aut(G,F)|} \hbar^{-\chi(G)} \\
&= \hbar \left( \frac{1}{|\Aut( {\begin{tikzpicture}[x=1ex,y=1ex,baseline={([yshift=-.5ex]current bounding box.center)}] \coordinate (vm); \coordinate [left=.7 of vm] (v0); \coordinate [right=.7 of vm] (v1); \draw (v0) circle(.7); \draw (v1) circle(.7); \filldraw (vm) circle (1pt); \end{tikzpicture}} )|} + \frac{1}{|\Aut( {\begin{tikzpicture}[x=1ex,y=1ex,baseline={([yshift=-.5ex]current bounding box.center)}] \coordinate (v0); \coordinate [right=1.5 of v0] (v1); \coordinate [left=.7 of v0] (i0); \coordinate [right=.7 of v1] (o0); \draw (v0) -- (v1); \filldraw (v0) circle (1pt); \filldraw (v1) circle (1pt); \draw (i0) circle(.7); \draw (o0) circle(.7); \end{tikzpicture}} )|} + \frac{1}{|\Aut( {\begin{tikzpicture}[x=1ex,y=1ex,baseline={([yshift=-.5ex]current bounding box.center)}] \coordinate (vm); \coordinate [left=1 of vm] (v0); \coordinate [right=1 of vm] (v1); \draw (v0) -- (v1); \draw (vm) circle(1); \filldraw (v0) circle (1pt); \filldraw (v1) circle (1pt); \end{tikzpicture}} )|} + \frac{(-1)}{|\Aut( {\begin{tikzpicture}[x=1ex,y=1ex,baseline={([yshift=-.5ex]current bounding box.center)}] \coordinate (vm); \coordinate [left=1 of vm] (v0); \coordinate [right=1 of vm] (v1); \draw[col1,thick] (v0) -- (v1); \draw (vm) circle(1); \filldraw (v0) circle (1pt); \filldraw (v1) circle (1pt); \end{tikzpicture}} )|} + \frac{(-1)}{|\Aut( {\begin{tikzpicture}[x=1ex,y=1ex,baseline={([yshift=-.5ex]current bounding box.center)}] \coordinate (v0); \coordinate [right=1.5 of v0] (v1); \coordinate [left=.7 of v0] (i0); \coordinate [right=.7 of v1] (o0); \draw[thick,col1] (v0) -- (v1); \filldraw (v0) circle (1pt); \filldraw (v1) circle (1pt); \draw (i0) circle(.7); \draw (o0) circle(.7); \end{tikzpicture}} )|} \right) +\ldots \\
&= -\frac{1}{24} \hbar - \frac{1}{48} \hbar^2 - \frac{161}{5760} \hbar^3 + \ldots \end{align*}
where we have colored the forest orange in the underlying graph.

\subsection{Counting trees}
A tree is a connected graph with no cycles; it has Euler characteristic equal to 1.  
For trees we will allow for the involution in its Definition~\ref{def:graph} as a graph to have fixed points. Each fixed point corresponds to a \emph{leave} of the tree.  We will denote the number of leaves of a tree $T$ by $\ell(T)$.
A \emph{rooted tree} is a tree with one univalent vertex, called the \emph{root} (note that in our definition of graph, a univalent vertex is not the same as a leaf).  
To extend our methods to trees, we will require   all vertices which are not the root  to have degree at least $3$, so that there are only a finite number of trees (or rooted trees) with a fixed number of leaves.

It is a classic combinatorial exercise to count rooted trees (see, e.g., \cite{Stanley}, Chapter~5.3). In this section we will solve the counting problem   that is relevant for the computation of the (virtual) Euler characteristic of $\Out(F_n)$.
Let $\mathcal R$ denote the set of isomorphism classes of rooted trees and define the generating function
\begin{equation}
\cR(\lambda,x)=\sum_{T_0\in\mathcal R} \frac{\lambda^{v(T_0)}x^{\ell(T_0)}}{|\Aut(T_0)|}.
\end{equation}
Eventually, we will only be interested in the special case $\lambda = -1$, which 
corresponds to the generating function of alternating sums over rooted trees with fixed number of leaves. 
But for the benefit of a concrete combinatorial argument, we will start with the general case and specialize later.

For example,   consider the following isomorphism classes of rooted trees
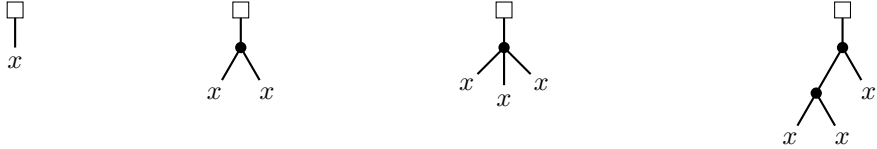
\newcommand{\rootedtree}[1]
{
  \draw[thick] (0,0) to (0,.5);
  \rootvertex{(0,.5)};
  \vertex{(0,0)};
  \foreach \x in {1,...,#1}
{
\pgfmathsetmacro\y{int(#1 + 1)}
\draw [thick] (0,0) to (-180*\x/\y:.5);
 \node  (x) at (-180/\y*\x:.7) {$x$};
}
}

\begin{center}\begin{tikzpicture} \draw[thick] (0,0) to (0,.5); \rootvertex{(0,.5)}; \node (x) at (0,-.2) {$x$}; \begin{scope}[xshift=3cm] \rootedtree{2} \begin{scope}[xshift=3.5cm] \rootedtree{3} \end{scope} \end{scope} \begin{scope}[xshift=11cm] \draw[thick] (0,0) to (0,.5); \rootvertex{(0,.5)}; \vertex{(0,0)}; \draw[thick] (0,0) to (-60:.5); \node (x) at (-60:.7) {$x$}; \draw[thick] (0,0) to (-120:.7); \vertex{(-120:.7)}; \draw[thick] (-120:.7) to ($(-120:.7) + (-120:.5)$); \node (x) at ($(-120:.7) + (-120:.7)$) {$x$}; \draw[thick] (-120:.7) to ($(-120:.7) + (-60:.5)$); \node (x) at ($(-120:.7) + (-60:.7)$) {$x$}; \end{scope} \end{tikzpicture}
\end{center}
where the root is indicated as a little square and
leaves are marked by $x$.
Note that the first rooted tree consists only of the root vertex and a single leaf.  In our definition of graph,  this is a single half-edge with trivial vertex partition and trivial involution.
The trees above contribute the following terms to $ \cR$:
 $$ \cR(\lambda,x) = x + \lambda\frac{x^2}{2!} + \lambda\frac{x^3}{3!}+ \lambda^2\frac{x^3}{2!}+  \ldots$$

Rooted trees have an inherently recursive structure. We can pick  such a tree up by the root and attach it to a leaf of another rooted tree. This gives a recursion which starts with the rooted tree that contains only the root vertex and a single half-edge.  %
This recursion is encoded by the following diagram, where the blue crosshatched circle indicates the set of all possibilities for the remainder of the tree:
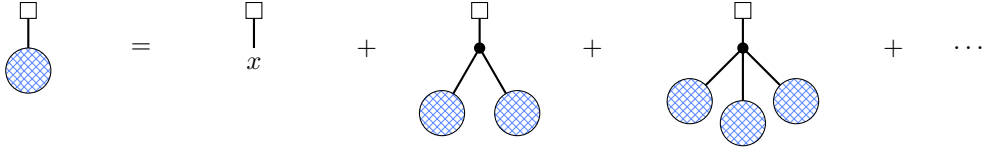
\newcommand{\rootedblob}[1]
{
  \draw[thick] (0,0) to (0,.5);
  \rootvertex{(0,.5)};
  \vertex{(0,0)};
  \foreach \x in {1,...,#1}
{
\pgfmathsetmacro\y{int(#1 + 1)}
\draw [thick] (0,0) to (-180*\x/\y:.7);
\draw [pattern=crosshatch, pattern color=col2]  (-180*\x/\y:1.0)   circle (.3);
}
}

\begin{center}\begin{tikzpicture} \draw[thick] (0,0) to (0,.5); \rootvertex{(0,.5)}; \draw [pattern=crosshatch, pattern color=col2] (0,-.3) circle (.3); \node (e) at (1.5,0) {$=$}; \begin{scope}[xshift=3cm] \draw[thick] (0,0) to (0,.5); \rootvertex{(0,.5)}; \node (x) at (0,-.2) {$x$}; \node (plus) at (1.5,0) {$+$}; \begin{scope}[xshift=3.cm] \node (plus) at (1.5,0) {$+$}; \rootedblob{2} \begin{scope}[xshift=3.5cm] \node (plus) at (2,0) {$+$}; \node (dots) at (3,0) {$\ldots$}; \rootedblob{3} \end{scope} \end{scope} \end{scope} \end{tikzpicture}
\end{center}

In terms of the generating function $ \cR(\lambda,x)$ we have
\begin{equation}\label{RTZ}
 \cR(\lambda,x)= x +  \lambda\frac{ \cR(\lambda,x)^2}{2!}+\lambda\frac{ \cR(\lambda,x)^3}{3!}+\ldots
= x+ \lambda(e^{ \cR(\lambda,x)}-1- \cR(\lambda,x)),
\end{equation}
where the factorials in the denominators account for permuting the   blue crosshatched circles.  
The power of $\lambda$ keeps track of the number of vertices in the tree, but for the Euler characteristic calculation we only care about the parity of $v$, i.e.\ we want to count rooted trees $T_0$ with sign $(-1)^{v(T_0)}$.  To accomplish this we can set $\lambda=-1$, which results in a neat closed expression for the generating function:
\begin{proposition} The generating function for rooted trees counted with sign $(-1)^{v(T_0)}$ is
\label{prop:Rsum}
\begin{align*} \cR(-1,x)=\sum_{T_0\in\mathcal R} \frac{(-1)^{v(T_0)}x^{\ell(T_0)}}{|\Aut(T_0)|}= \log(1+x) = \sum_{\ell\geq 1} \frac{(-1)^{\ell+1}}{\ell} x^\ell. \end{align*}
\end{proposition}

\begin{proof}
By Equation~(\ref{RTZ}) we have
$$ \cR(-1,x)=x-e^{ \cR(-1,x)}+1+ \cR(-1,x)$$
i.e.\ $e^{ \cR(-1,x)}=x+1.$
\end{proof}

Let $\mathcal T(\ell,k)$ (resp. $\mathcal {R}(\ell,k)$) be the set of isomorphism classes of  trees (resp. rooted trees) with $\ell$ leaves and $k$ vertices and $\mathcal{LT}(\ell,k)$ (resp. $\mathcal {LR}(\ell,k)$) be the corresponding set of leaf-labeled   trees (resp. rooted trees), i.e.\ each leaf is labeled by a number in $1,\ldots,\ell$.
The symmetric group $\Sigma_\ell$ acts on the set of leaf-labels, with orbit corresponding to an isomorphism class of trees $T$ and stabilizer $\Aut(T)$, so the orbit-stabilizer theorem  gives
\begin{gather} \label{eq:orbtree} \sum_{ T \in \mathcal T(\ell,k)} \frac{1}{|\Aut (T)|} = \frac{|\mathcal{LT}(\ell,k)|}{\ell!} \hbox{ and } \sum_{ T_0 \in \mathcal R(\ell,k)} \frac{1}{|\Aut (T_0)|} = \frac{|\mathcal{LR}(\ell,k)|}{\ell!}. \end{gather}

\begin{proposition}  When counted with sign $(-1)^{v(T_0)}$ there are 
 $(-1)^{\ell+1}(\ell-1)!$ leaf-labeled rooted trees $T_0$ with $\ell\geq 1$ leaves.
\end{proposition}

\begin{proof}
Let $\mathcal{LR}(\ell)$ be the set of leaf-labeled rooted trees with $\ell \geq 1$ leaves. We are interested in,
$$
\sum_{T_0\in\mathcal{LR}(\ell)} (-1)^{v(T_0)}
=
\sum_{k \geq 0} (-1)^k  |\mathcal{LR}(\ell,k)|.
$$
Combining~\eqref{eq:orbtree} and Proposition~\ref{prop:Rsum} gives 
\begin{gather*} \sum_{k\geq 0}(-1)^k|\mathcal{LR}(\ell,k)|=\ell!\sum_{k\geq 0} \sum_{T_0\in \mathcal{R}(\ell,k)} \frac{(-1)^k}{|\Aut(T_0)|}=\ell!\frac{(-1)^{\ell+1}}{\ell}=(-1)^{\ell+1}(\ell-1)!. \qedhere \end{gather*}
\end{proof}

We next count unrooted trees with labeled leaves.  Note that an unrooted tree must have at least three leaves, since we do not allow univalent or bivalent vertices.

\begin{proposition}  When counted with sign $(-1)^{v(T)}$ there are 
 $(-1)^{\ell}(\ell-2)!$ leaf-labeled trees $T$ with $\ell\geq 3$ leaves.
\end{proposition}

\begin{proof}  An (unrooted) tree with $\ell$ leaves labeled $1,\ldots,\ell$ is equivalent to a rooted tree with $\ell-1$ leaves, where the leaf labeled $\ell$ is taken for the root (or, equivalently, an extra half-edge is added which is matched with the half-edge labeled $\ell$ and forms its own block of the vertex partition).  By the above lemma, there are $(-1)^{\ell}(\ell-2)!$ such trees, counted with sign $(-1)^{v(T)}$. 
\end{proof}

Let $\cT(x)$ be the generating function for signed (unrooted) trees $T$, with $x$ marking leaves, i.e.
$$ \cT(x) =
\sum_{T\in \mathcal T}
(-1)^{v(T)}
\frac{x^{\ell(T)}}{|\Aut (T)|}
,$$
where the  sum is over the set of all isomorphism classes of trees $\mathcal T$.

By~\eqref{eq:orbtree} and the above Corollary,  
$$\cT(x) =
\sum_{\ell,k \geq 0}
(-1)^k
|\mathcal{LT}(\ell,k)|
\frac{x^\ell}{\ell!}
=
\sum_{\ell \geq 3}
(-1)^{\ell}(\ell-2)!
\frac{x^\ell}{\ell!}
=
\sum_{\ell \geq 3}
(-1)^{\ell}
\frac{x^\ell}{\ell(\ell-1)}
$$
It is easy to check that the last expression agrees with the power series expansion of 
$$\cT(x) =
-x-\frac{x^2}{2} + (1+x) \log(1+x).
$$

Our passing from the generating function $\cR(-1,x)$ of signed rooted trees  to the 
generating function $\cT(x)$ of signed trees  has an interpretation in terms 
of a \emph{Legendre transformation} that is also of significance in the QFT context (see, e.g., \cite{Borinsky}, Chapter~5.8).

\subsection{Counting graphs with a marked forest}
A spanning forest $F$ in a graph $G$ has the same vertices and half-edges as $G$ but only some of the edges; in particular it is not allowed to contain a cycle of edges.   
To count signed pairs $(G,F)$ with $G$ admissible we will first count labeled pairs as before, but we will only label the half-edges that are not contained in edges of $F$, i.e. these are the leaves of $F$.  Notice that $F$ must have an even number $2\ell$ of leaves since $G$ is obtained by matching pairs of them, and $G$ itself has no leaves.     

 Let $\mathcal{LP}(t,2\ell,m)$ be the set of labeled pairs $(G,F)$ such that $F$ has $t$ trees with a total of $2\ell$ leaves and $m$ edges.  The symmetric group  $\Sigma_{2\ell}$ acts on such pairs by permuting the labels.  The orbit of a labeled pair is the unlabeled pair $(G,F)$, and the stabilizer is $\Aut(G,F)$, the automorphisms of $G$ that preserve $F$.  We now run the argument exactly as we did when counting graphs in Sections~\ref{unlabeled} and \ref{labeled}:
The orbit-stabilizer theorem gives
$$|\mathcal{LP}(t,2\ell,m)|=\sum_{{\rm orbits}} |{{\rm orbit}}| = \sum_{(G,F)\in\mathcal{P}(t,2\ell,m)} \frac{ (2\ell)!}{|\Aut(G,F)|},$$
where $\mathcal P(t,2\ell,m)$ is the set of unlabeled pairs $(G,F)$.

For each tree $T$ we have $e(T)-v(T) =1$. Therefore, in order to count forests with sign $(-1)^{e(F)}$, we will use the function 
$$ -\cT(x) =
-\sum_{T\in \mathcal T}
(-1)^{v(T)}
\frac{x^{\ell(T)}}{|\Aut (T)|}
=
\sum_{T\in \mathcal T}
(-1)^{e(T)}
\frac{x^{\ell(T)}}{|\Aut (T)|}.
$$ 
In order to keep track of how many trees there are in a forest we mark their number by $\tau$ and take $\exp(-\tau \cT(x))$.
Since a labeled pair is obtained by taking a forest with $t$ trees and matching their $2\ell$ leaves, we have
$$\frac{1}{(2\ell)! }\sum_{m \geq 0} (-1)^{m} |\mathcal{LP}(t, 2\ell,m)|=(2\ell-1)!!\,[\tau^t x^{2\ell}]\exp(-\tau\cT(x)).$$
For each labeled pair $(G,F) \in \mathcal{LP}(t,2\ell,m)$  the forest $F$ has $t$ trees, so we have  $v(F)-e(F)=t$ and $-\chi(G)=\ell-t$. Therefore
\begin{equation}\label{count}
\sum_{(G,F)}(-1)^{e(F)}\frac {\hbar^{-\chi(G)}}{|\Aut(G,F)|}=\sum_{\ell\geq 0} \hbar^{\ell} (2\ell-1)!![x^{2\ell}]\exp(-\hbar^{-1}\cT(x)),
\end{equation}
where we sum over all isomorphism classes of pairs $(G,F)$.

Because $\exp(-\hbar^{-1}\cT(x))$ counts leaf-labeled forests,  Equation~\eqref{count} can be seen as a special case of Proposition~4.7 in \cite{BoVo}.

Using the exponential formula, we finally obtain an 
expression for the generating function $\cY(\hbar)$ of $\chi(\Out(F_n))$:
$$
\exp (\cY(\hbar)) = 
\sum_{\ell\geq 0} \hbar^{\ell} (2\ell-1)!![x^{2\ell}]\exp(-\hbar^{-1}\cT(x)).
$$
As in the previous sections, we can also write this formally as an integral expression:
\begin{align} \begin{aligned} \label{Yintegral} \exp (\cY(\hbar)) &= \int \frac{\dd x}{\sqrt{2 \pi \hbar}} \exp \left( \hbar^{-1}( x - (1+x) \log (1+x) ) \right) \\
&= \int \frac{\dd x}{\sqrt{2 \pi \hbar}} \exp \left( - \hbar^{-1} \sum_{n \geq 3} (-1)^{n} \frac{x^n}{n(n-1)} \right). \end{aligned} \end{align}
In QFT terminology, we have shown that the free energy of the 
$0$-dimensional QFT with action $x - (1+x) \log (1+x)$ 
encodes the topological invariants $\chi(\Out(F_n))$.

Note that expression~\eqref{Yintegral} agrees with Equation~(1) of \cite{Kont} up to an irrelevant sign in front of $x$. The reason for this sign difference is that we computed the virtual Euler characteristic of $\Out(F_n)$ using the spine of Outer space, whereas Kontsevich computes the Euler characteristic of the Lie graph complex, which is only quasi-isomorphic to the chain complex of the spine.

\subsection{A renormalized topological QFT}
In~\cite{BoVo} we obtained different expressions for the invariants $\chi(\Out(F_n))$ (Theorem B and Proposition 8.1) that enabled us to prove that $\chi(\Out(F_n)) < 0$ for all $n \geq 2$ and that their growth rate is given by $\chi(\Out(F_n)) \sim -\tfrac{1}{\sqrt{2\pi}} \Gamma(n-\tfrac32)/\log^2n$ for $n \rightarrow \infty$ (Theorem A).

Key intuition for proving these facts came from the observation that the invariants $\chi(\Out(F_n))$ are also encoded by a \emph{renormalization} of the QFT described by Equation~\eqref{Epathintegral} in Section~\ref{secgc2}.  
Specifically, the generating function $\cY(\hbar)$ also respects the implicit path-integral expression
\begin{align*} 1 &= \frac{1}{\sqrt{2\pi \hbar}} \int \exp\left( -\hbar^{-1}(e^x-x-1) + \frac{x}{2} + \cY(-\hbar e^{-x}) \right) dx. \end{align*}
This equation fixes the values of the invariants $\chi(\Out(F_n))$ and enables one  to study them.
To fully justify this, a precise definition of the renormalization procedure 
in QFT is needed, but this lies beyond the scope of these expository notes (see \cite{BoVo}~Section~6 for a derivation).

The QFT interpretation of the above equation, however,  is readily stated: The generating function
$\frac{x}{2} + \cY(-\hbar e^{-x})$ encodes  \emph{counterterms} or \emph{renormalization constants} that renormalize 
the QFT with action $-(e^x-x-1)$, which was the subject of Section~\ref{secgc2}, in a generalized sense. See~\cite{Borinsky} Chapters~5--7 for a detailed algebraic and combinatorial treatment of such 
renormalization operations.

\newcommand*{\doilink}[2]{\href{https://doi.org/\detokenize{#2}}{#1}}
\newcommand*{\arxivlink}[1]{\href{https://arxiv.org/abs/\detokenize{#1}}{arXiv:#1}}

\end{document}